\documentclass[12pt]{amsart}

\usepackage{psfrag}
\usepackage{amsmath}
\usepackage{amssymb}
\usepackage{epsfig}
\usepackage{amsfonts}

\newcommand{\excise}[1]{}%$\star$\textsc{#1}$\star$}

\newtheorem{thm}{Theorem}[section]
\newtheorem{lemma}[thm]{Lemma}

\newtheorem{conv}[thm]{Convention}

\newtheorem{Warn}[thm]{Caution}

    {\end{Example}}
    {\end{Remark}}

\def\RP{\mathbb R \mathbb P}

\def\rr{\mathbb R}

\def\De{\Delta}

\def\ssu{\subset}

\def\<{\langle}
\def\>{\rangle}

\def\0{{\mathbf 0}}

\def\.{\hskip.06cm}
\def\ts{\hskip.03cm}

\begin{document}
\title{The discrete square peg problem}

\author[Igor~Pak]{ \ Igor~Pak$^\star$}
\date{April 3, 2008}

%\excise{\keywords{Convex polytopes, rigidity, polynomial invariants}}

\thanks{\thinspace ${\hspace{-.45ex}}^\star$School of Mathematics,
University of Minnesota, Minneapolis, MN, 55455. \.
Email: \ts \texttt{pak@umn.edu}}

\maketitle

% {\hskip6.05cm
% April 2, 2008
% }

\vskip1.3cm

\begin{abstract}
The square peg problem asks whether every Jordan curve in the plane
has four points which form a square.  The problem has been resolved
(positively) for various classes of curves, but remains open in full
generality.  We present two new direct proofs for the case of
piecewise linear curves.
\end{abstract}

\vskip2.5cm

\section*{Introduction}\label{intro}

\noindent
The \emph{square peg problem} is beautiful and deceptively simple.
It asks whether every Jordan curve $C \ssu \rr^2$
has four points which form a square.  We call such
squares \emph{inscribed} into~$C$ (see Figure~\ref{f:inscribe-babochka}).

\begin{figure}[hbt]
\psfrag{C}{$C$}
\begin{center}
\epsfig{file=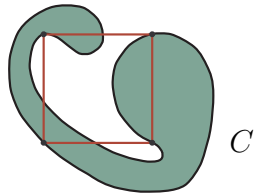,width=3.2cm}
\end{center}
\caption{Jordan curve~$C$ and an inscribed square. }
\label{f:inscribe-babochka}
\end{figure}

The problem goes back to Toeplitz (1911), and over almost a century has
been repeatedly rediscovered and investigated, but never completely resolved.
By now it has been established for convex curves and curves with various
regularity conditions, including the case of piecewise linear curves.
While there are several simple and elegant proofs of the convex case,
the piecewise linear case is usually obtained as a consequence of
results proved by rather technical topological and analytic arguments.
In fact, until to this paper, there was no direct elementary proof.
Here we present two such proofs in the piecewise linear case.

\medskip

\noindent
{\bf Main Theorem.} \,
\emph{Every simple polygon on a plane has an inscribed square.
}

\medskip

As the reader will see, both proofs are direct and elementary, although perhaps
not to the extend one would call them ``book proofs''.  The proofs are strongly
motivated by the classical ideas in the field (see Section~\ref{s:history}).
Here and there, we omit a number of minor straightforward details,
in particular the deformation construction at the end of the second proof.

The rest of the paper is structured as follows.  In the next two sections we
present the proofs of the main theorem.  These proofs are completely separate
and can be read independently.  In the last section we give an outline of the
rich history of the problem and the underlying ideas.  The historical part
is not meant to be comprehensive, but we do include a number of pointers
to surveys and recent references.

\bigskip

\section{Proof via inscribed triangles} \label{s:triangles}

\noindent
Let $X=[x_1\ldots x_n] \in \rr^2$ be a simple polygon.  We assume
that~$X$ is generic in a sense which will be clear later on.
Further, we assume that the angles of~$X$ are \emph{obtuse},
i.e.~lie between $\pi/2$ and~$3\ts \pi/2$.  Fix a clockwise
orientation on~$X$.

For an ordered pair $(y,z)$ of points $y,z \in X$ denote by~$u$
and~$v$ the other two vertices of a square $[zyuv]$ in the plane,
with vertices on~$X$ in this order, as shown in
Figure~\ref{f:inscribe-sqjerrard}.
Parameterize~$X$ by the length and think of $(y,z)$ as a point on a torus
$T = X\times X$.  Denote by $U \ssu T$ the subset of pairs~$(y,z)$
so that $u \in X$.  Similarly, denote by $V \ssu T$ the subset of
pairs~$(y,z)$ so that~$v \in X$.  Our goal is to show that~$U$
intersects~$V$.

\begin{figure}[hbt]
\psfrag{y}{\small $y$}
\psfrag{z}{\small $z$}
\psfrag{u}{\small $u$}
\psfrag{v}{\small $v$}
\psfrag{X}{\small $X$}
\psfrag{X'}{\small $X'$}
\begin{center}
\epsfig{file=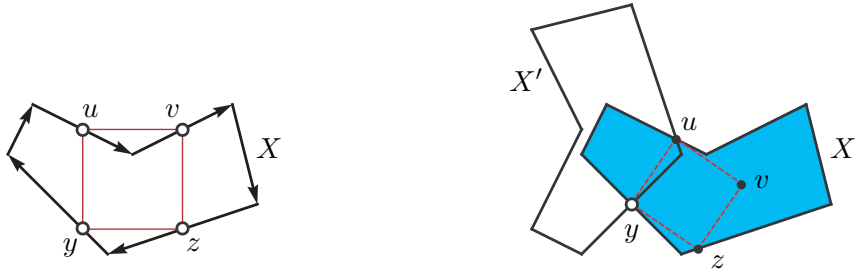,width=11.5cm}
\end{center}
\caption{Square $[zyuv]$ inscribed into~$X$ and a rotation~$X'$ of~$X$ around~$y$. }
\label{f:inscribe-sqjerrard}
\end{figure}

First, observe that for a generic~$X$ the set of points
$U_y = \{z\ts: \ts (y,z) \in U\}$ is finite.
Indeed, these points~$z \in U_y$ lie in the intersection of the
polygon~$X$ and a polygon~$X'$ obtained by a counterclockwise
rotation of~$X$ around~$y$ by an angle~$\pi/2$
(see Figure~\ref{f:inscribe-sqjerrard}).
Therefore, if~$X$ does not have orthogonal edges there is
only a finite number of points in $X \cap X'$.  Moreover, when~$y$
moves along~$X$ at a constant speed, these intersection points~$z$
change piecewise linearly, which implies that~$U$ is also piecewise
linear.

Let us show that~$U$ is a disjoint union of simple polygons.  Observe
that when~$y$ is moved along~$X$ the intersection point $z \in X \cap X'$
cannot disappear except when a vertex of~$X$ passes through an edge of~$X'$,
or when a vertex of~$X'$ passes through an edge of~$X$.
This implies that when~$y$ is moved along~$X$ the intersection points
emerge and disappear in pairs, and thus~$U$ is a union of polygons.
Note that for a generic~$X$, at no time can a vertex pass through
a vertex, which is equivalent to the condition that no square with a
diagonal~$(x_i,x_j)$ can have a point $y \in X$ as its third vertex.

To see that the polygons in~$U$ are simple and disjoint, observe that the
only way we can have an intersection if a vertex of~$X'$ changes direction
at an edge in~$X$, or, similarly, if a vertex of~$X$ changes direction
at an edge in~$X'$.  This is possible only when~$y$ and either~$z$ or~$u$
are vertices of~$X$.  Since~$X$ is chosen to be generic we can assume
this does not happen, i.e.~that $X$ does not have an inscribed right
isosceles triangle with an edge~$(x_i,x_j)$.

\begin{figure}[hbt]
\psfrag{X}{\footnotesize $X$}
\psfrag{X'}{\footnotesize $X'$}
\begin{center}
\epsfig{file=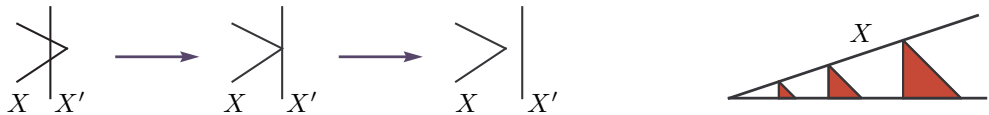,width=13.2cm}
\end{center}
\caption{Two disappearing points in $X \cap X'$ and a converging family
of right isosceles triangles inscribed into~$X$ with angles $<\pi/2$. }
\label{f:inscribe-sq-disappear}
\end{figure}

A similar argument also implies that on a torus~$T$, the set~$U$
separates points $(y,z) \in T$ with the corresponding vertex~$u$
inside of~$X$, from those where~$u$ is outside.  By continuity,
it suffices to show that the point~$u$ crosses the edge of~$X$
as the generic point $(y,z)$ crosses~$U$.  Consider a point
$(y,z) \in U$ such that the corresponding third vertex of a
square~$u$ lies in the relative interior of an edge~$e$ in~$X$.
Now fix~$y$ and change~$z$.  Since~$X$ is generic, point~$u$
will pass through the edge~$e$, which implies the claim.

We need a few more observations on the structure of~$U$.  First, observe
that~$U$ does not intersect the diagonal $\De = \{(y,y), y \in X\}$.
Indeed, otherwise we would have a sequence of inscribed right triangles
$(y,z,u)$ converging to the same point, which is impossible since~$X$
does not have angles between $\pi/2$ and $3\ts\pi/2$
(see Figure~\ref{f:inscribe-sq-disappear}).
In a different direction, observe that for a generic~$y$ the number
of points in~$U_y$ is odd.  This follows from the previous argument
and the fact the number of intersections of~$X$ and~$X'$ as even
except at a finite number of points~$y$.

Now, from above we can conclude that at least one of the polygons in~$U$
is not null homotopic on the torus~$T$, since otherwise for a generic
point~$y$ the size of~$U_y$ is even.  Fix one such polygon and denote
it by~$U^\circ$. Since~$U^\circ$ is simple, not null homotopic and does
not intersect the diagonal~$\De$, we conclude that~$U^\circ$ is homotopic
to~$\De$ on~$Y$.  Therefore, there exist a continuous family of inscribed
right isosceles triangles $(uyz)$ such that when~$y$ goes around~$X$ so
do~$z$ and~$u$.  Relabeling triangles $(uyz)$ with $(yzv)$ we obtain
a simple polygon $V^\circ \ssu V$ which is also homotopic to~$\De$ on~$T$.

\begin{figure}[hbt]
\psfrag{1}{\small $A$}
\psfrag{2}{\small $B$}
\psfrag{3}{\small $C$}
\psfrag{V}{\small $V^\circ$}
\psfrag{W}{\small $U$}
\psfrag{U}{\small $U^\circ$}
\psfrag{T}{\, \small $T$}
\psfrag{D}{\small $\De$}
\begin{center}
\epsfig{file=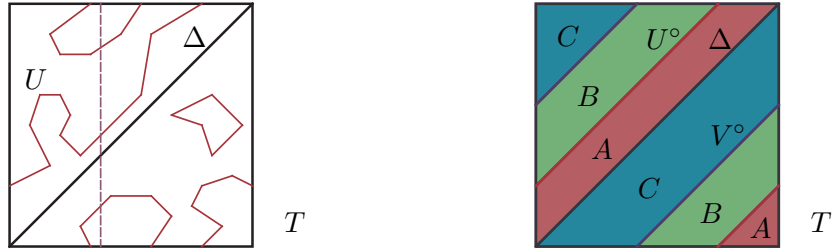,width=10.8cm}
\end{center}
\caption{Set~$U$ on a torus~$T$ and the sequence of regions $A,B,C \ssu T$. }
\label{f:inscribe-jr-regions}
\end{figure}

Suppose now that~$U^\circ$ and~$V^\circ$ do not intersect.  Together
with~$\De$ these curves separate the torus~$T$ into three regions:
region~$A$ between $\De$ and~$U^\circ$, region~$B$ between~$U^\circ$ and
$V^\circ$, and region~$C$ between $V^\circ$ and~$\De$
(see Figure~\ref{f:inscribe-jr-regions}).  Consider the pairs $(y,z)$
in the regions~$A$ and~$C$ which lie close to~$\De$ (i.e.~$y$ and~$z$ lie
close to each other on~$X$).
Clearly, for such $(y,z)$, either both corresponding points~$u$ and~$v$
lie \emph{inside}~$X$ or both~$u$ and~$v$ lie \emph{outside} of~$X$.
Let~$A$ be the former and let~$C$ be the latter regions.
From above, for all $(y,z) \in B$ we have $u \notin X$ and $v \in X$.
In other words, when~$y$ is fixed and~$z$ is moved along~$X$
counterclockwise starting at~$y$, of the points~$u$ and~$v$ the first
to move outside of~$X$ is always~$u$.

Now, consider the smallest right equilateral triangle~$R$ inscribed
into~$X$ (the existence was shown earlier).   There are two ways to
label it as shown in Figure~\ref{f:inscribe-jr-smallest}.  For the
first labeling, if~$y$ is fixed and~$z$ is moved as above, the first
time point~$u$ lies on~$X$ is when~$z$ and~$u$ are vertices of~$R$.
By assumptions on region $A \ssu T$, the corresponding point~$v$ lies
inside~$X$.  Similarly, for the
second labeling, if~$y$ is fixed and~$z$ is moved as above, the first
time point~$v$ lies on~$X$ is when~$z$ and~$v$ are vertices of~$R$.
By assumptions on regions $A \ssu T$, the corresponding point~$u$ lies
outside of~$X$, a contradiction.\footnote{The figure is somewhat misleading
as it gives the impression that for all $y$ and~$z$ with~$|yz|$ smaller than
that in~$R$,
we must have $(y,z) \in A$.  In fact, we can have all these pairs in~$C$
and the same argument will work when~$A$ is substituted with~$C$ and the
inside/outside properties are switched accordingly.  The point is, by
continuity, all close~$(y,z)$ with a fixed order on~$X$ determined by~$R$,
must lie in the same region (either~$A$ or~$C$).}

\begin{figure}[hbt]
\psfrag{y}{\small $y$}
\psfrag{z}{\small $z$}
\psfrag{u}{\small $u$}
\psfrag{v}{\small $v$}
\psfrag{X}{\small $X$}
\psfrag{R}{\small $R$}
\begin{center}
\epsfig{file=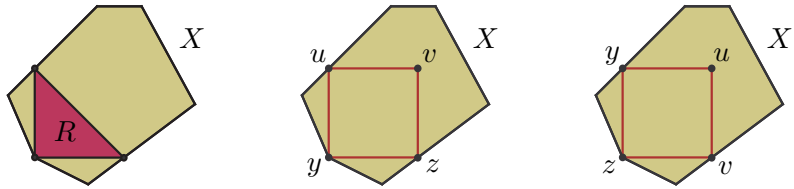,width=10.5cm}
\end{center}
\caption{The smallest inscribed right isosceles triangle and its two labelings. }
\label{f:inscribe-jr-smallest}
\end{figure}

Finally, suppose~$X$ is not generic.  We can
perturb the vertices of~$X$ to obtain a continuous family of generic
polygons converging to~$X$ and use the limit argument.  Since~$X$ is
simple, the converging squares do not disappear and converge to a
desired inscribed square.  Similarly, when~$X$ has angles~$<\pi/2$ or
$>3\pi/2$, use the limit argument by cutting the corners as shown
in Figure~\ref{f:inscribe-jr-corner}.

\begin{figure}[hbt]
\begin{center}
\epsfig{file=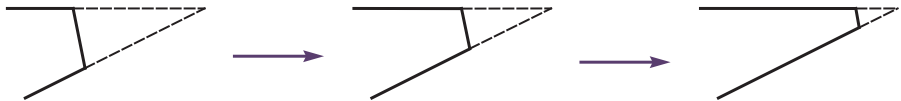,width=12.1cm}
\end{center}
\caption{A converging family of obtuse polygons. }
\label{f:inscribe-jr-corner}
\end{figure}

\bigskip

\section{Proof by deformation}\label{s:deform}

\noindent
In this section, we prove the following extension of the main theorem:
every generic simple polygon has an odd number of inscribed squares.
Now that we have the relation, we can try to prove that it is invariant
under certain elementary transitions.

\begin{thm} \label{t:deform-square}
Every generic simple polygon has an odd number of inscribed
squares.\footnote{It takes some effort to clarify what we mean by a generic
(see the proof).  For now, the reader can read this as saying that the
$n$-gons, viewed as points in~$\rr^{2n}$, are almost surely generic.}
\end{thm}

The main theorem now follows by a straightforward limit
argument.  Note also that the theorem is false for \emph{all} simple polygons;
for example every right triangle has exactly two inscribed squares.
We begin the proof with the following simple statement.

\begin{lemma} \label{l:deform-square-lines}
Let $\ell_1,\ell_2,\ell_3$ and~$\ell_4$ be four lines in~$\rr^2$ in
general position.  Then there exists a unique square $A=[a_1a_2a_3a_4]$
such that $x_i \in \ell_i$ and~$A$ is oriented clockwise. \ts
Moreover, the map $\ts (\ell_1,\ell_2,\ell_3,\ell_4) \to (a_1,a_2,a_3,a_4)\ts $
is continuously differentiable, where defined.\footnote{To make this precise,
think of lines~$\ell_i$ as points in $\RP^2$.}
\end{lemma}

\begin{proof}  Fix $z_1 \in \ell_1$.  Rotate~$\ell_4$ around~$z_1$
by~$\pi/2$, and denote by~$\ell_4'$ the resulting line, and by
$z_2 = \ell_2\cap \ell_4'$ the intersection point.  Except when
$\ell_2 \bot \ell_4$, such~$z_2$ is unique.  Denote by~$z_4 \in \ell_4$
the inverse rotation of~$z_2$ around~$z_1$.  We obtain the right isosceles
triangle $\De=(z_2z_1z_4)$ oriented clockwise in the plane.
The fourth vertex~$z_3$ of a square is uniquely determined.
Start moving~$z_1$ along~$\ell_1$ and observe that the locus of~$z_3$
is a line, which we denote by~$\ell_3'$.  Since line~$\ell_3$ is in
general position with respect to~$\ell'_3$, these two line intersect
at a unique point~$x_3$, i.e.~determines uniquely the
square~$[a_1a_2a_3a_4]$ as in the theorem. \. The second part
follows from immediately from the above construction.
\end{proof}

\medskip

\begin{proof}[Sketch of proof of Theorem~\ref{t:deform-square}]
We begin with the following restatement of the second part of the lemma.
Let $X=[x_1\ldots x_n]$ be a generic simple polygon and let
$\{X_t, t\in [0,1]\}$  be its continuous piecewise linear deformation.
Suppose $A=[a_1a_2a_3a_4]$ is an inscribed square with vertices~$a_i$
at different edges of~$X$, and none at the vertices of~$X$,
i.e.~$a_i \ne x_j$.  Then, for sufficiently small~$t$, there
exists a continuous deformation $\{A_t\}$ of inscribed squares,
i.e.~squares $A_t$ inscribed into~$X_t$.  Moreover, for
sufficiently small~$t$, the vertices $a_i$ of~$A_t$ move
monotonically along the edges of~$X_t$.

Consider what can happen to inscribed squares~$A_t$ as~$t$ increases.
First, we may have some non-generic polygon~$X_s$, where such square
in non-unique or undefined.  Note that the latter case is impossible,
since by compactness we can always define a limiting square~$A_s$.
If the piecewise linear deformation~$\{X_t\}$
is chosen generically, it is linear at time~$s$, and
we can extend the deformation of~$A_t$ beyond~$A_s$.

The second obstacle is more delicate and occurs when the vertex~$a_i$
of square~$A_s$ is at a vertex~$x_j$ of~$X_s$. Clearly, we can no
longer deform~$A_s$ beyond this point.  Denote by~$e_1$ the edge of~$X$
which contains vertices~$a_i$ of $A_t$ for $t<s$.  Clearly,
$e_1=(x_{j-1},x_{j})$ or $e_1=(x_j,x_{j+1})$.  Denote by~$e_1'$ the
other edge adjacent to~$v$.  Denote by $e_2,e_3$ and~$e_4$ the other
three edges of~$X$ containing vertices of~$A_t$ (see Figure~\ref{f:deform-sq-move}).

\begin{figure}[hbt]
\psfrag{A}{\small $A_t$}
\psfrag{B}{\small $A_s$}
\psfrag{C}{\small $B_r$}
\psfrag{x}{\footnotesize $x_j$}
\psfrag{e}{\footnotesize $e_1$}
\psfrag{e'}{\footnotesize $e_1'$}
\psfrag{e2}{\footnotesize $e_2$}
\psfrag{e3}{\footnotesize $e_3$}
\psfrag{e4}{\footnotesize $e_4$}
\begin{center}
\epsfig{file=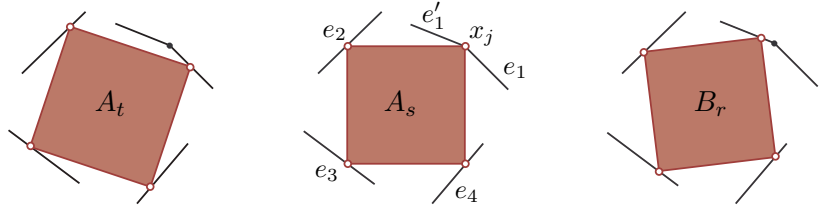,width=10.9cm}
\end{center}
\caption{Inscribed squares $A_t$, $A_s=B_s$ and $B_r$, where
$t < s <r$. Here~$e_2,e_3$ and~$e_4$ are fixed,
while~$e_1$ and~$e_1'$ move away from the squares.}
\label{f:deform-sq-move}
\end{figure}

Now consider a family $\{B_t\}$ of squares inscribed into lines
spanned by edges~$e_1',e_2,e_3$ and~$e_4$. By construction, $A_s = B_s$.
There are two possibilities: either the corresponding vertex~$b_i$
approaches~$x_j$ from inside~$e_1'$ or from the outside, when $t \to s$ and~$t <s$.
In the former case, we conclude that the number of inscribed squares
decreases by~2 as~$t$ passes through~$s$.  In the latter case, one square
appears and one disappears, so the parity of the number of squares remains
the same.  In summary, the parity of the number of squares inscribed
into~$X_t$ with vertices at different edges is invariant under the
deformation.

It remains to show that one can always deform the polygon~$X$ in such
a way that at no point in the deformation do there exist inscribed
squares with more than one vertex at the same edge, and such that the
resulting polygon has an odd number of inscribed squares.

Fix a triangulation $T$ of~$X$.  Find a triangle~$\De$
in~$T$ with two edges the edges of~$X$ and one edge a diagonal in~$X$.
Subdivide the edges of~$X$ into small edges, so that neither of the new
vertices is a vertex of an inscribed square.  If the edge length is now
small enough, we can guarantee that no square with two vertices at the
same edge is inscribed into~$X$.  Now move the edges along two sides
of the triangle~$\De$ toward the diagonal as shown in
Figure~\ref{f:deform-square}.  Repeat the procedure.  At the
end we obtain a polygon~$Z$ with edges close to an interval.  Observe
that~$Z$ has a unique inscribed square
(see Figure~\ref{f:deform-square}).  This finishes the proof.
\end{proof}

\begin{figure}[hbt]
\psfrag{X}{\small $X$}
\psfrag{Z}{\small $Z$}
\psfrag{T}{\small $\De$}
\begin{center}
\epsfig{file=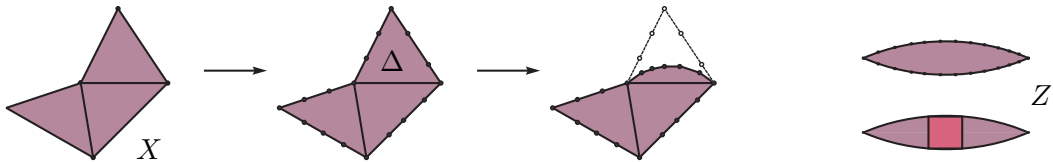,width=13.96cm}
\end{center}
\caption{The first step of the polygon deformation which preserves the parity of the
number of inscribed squares; the final polygon~$Z$. }
\label{f:deform-square}
\end{figure}

\bigskip

\section{The history, the proof ideas and the final remarks} \label{s:history}

%\noindent
\subsection{} \.
The square peg problem of inscribed squares has a long and interesting history.
It seems, every few years someone new falls in love with it and works very
hard to obtain a new variation on the problem.  Unfortunately, as the results
become stronger, the solutions become more technically involved and several of
them start to include some gaps, still awaiting careful scrutiny.\footnote{While
we did find some such rather unconvincing arguments, we refrain from commenting on
them and leave them to the experts.}
Interestingly, the impression one gets from the literature is that that no direct
elementary proof is even possible in the piecewise linear case, as the problem
is difficult indeed, the existing techniques are inherently non-discrete and,
presumably, other people have tried.

\subsection{} \.
We begin with the celebrated incorrect proof by Ogilvy~\cite{FO}.   While the
proof was refuted by several readers within a few months after its publication,
it is still worth going over this proof and try find the gaps (there are three
major ones, even if one assumes that the curve is piecewise linear or analytic).
As reported in~\cite{KW}, Ogilvy later became disillusioned in the possibility
of a positive resolution of the problem.

\subsection{} \.
The first major result was proved by Emch, who established the square peg
problem for convex curves~\cite{Emch}.  Later, Emch writes in~\cite{Emch-medians}
that Toeplitz and his students discovered the result independently two years
earlier, in~1911, but never published the proof. We refer to~\cite[p.~84]{Gru-cbsm}
for further references to proofs in the convex case).  Emch starts by constructing
a family of inscribed rhombi with a diagonal parallel to a given line.  By rotating
the line and using uniqueness of such rhombi he concludes that one can
continuously rotate a rhombus into itself with two diagonals interchanged.
Then the intermediate value theorem implies that at some point the rhombus
has equal diagonals, thus giving a square.

In the largely forgotten followup paper~\cite{Hebbert}, Hebbert studies the
squares inscribed into quadrilaterals, essentially proving
Lemma~\ref{l:deform-square-lines}.  He stops short of applying his
observations to general simple polygons.
Let us mention also that in the second proof we use the fact that
every non-convex polygon can be triangulated.  This is a standard
result also due to Emch~\cite{Emch-medians}.

\subsection{} \.
An important breakthrough was made by Shnirelman in~1929, when he offered a
solution for curves with piecewise continuous curvature.  This paper was
published in an obscure Russian publication, but later an expanded
version~\cite{Shn} was published posthumously.  Guggenheimer in~\cite{Gugg-sh}
studied this proof, added and correct several technical points, and concluded
that for Shnirelman's proof to work the curve needs to have a bounded variation.
Shnirelman noted that for a generic curve the parity of the number of inscribed
squares must be invariant as the curve is deformed.  The proof uses a local
lemma on existence of inscribed square for closed curves, a non-linear version
of Hebbert's observation (and, most likely, completely independent).  For the
connectivity of curves with continuous curvature and bounded variation
Shnirelman and Guggenheimer use known advanced results in the field.  Finally,
the fact that every ellipse with unequal axis has a unique inscribed square
is straightforward.

Our proof in Section~\ref{s:deform} is modeled on the deformation idea of Shnirelman
(we were unaware of Hebbert's paper).  In the piecewise linear case we no longer have
the analytic difficulties, but we do get the unpleasant obstacle of having inscribed
squares with more than two vertices on the same edge.  In fact, if not for the smooth
case, there is no intuitive reason behind Theorem~\ref{t:deform-square}.

Interestingly, we believe we know where Shnirelman got the idea of his proof.
At the time of his first publication, Shnirelman was working with Lyusternik
on the conjecture of Poincar{\'e} which states that every smooth convex surface
has at least three closed geodesics.  This conjecture was made in the
foundational paper~\cite{P}, where Poincar{\'e} proves that at least one
such closed geodesics exists (on analytic surfaces), and this proof uses
a similar deformation and parity argument.

\subsection{} \.
In 1961, Jerrard rediscovered the square peg problem and proved it for
analytic curves.  He was apparently motivated by the Kakutani's theorem
that every convex body has a circumscribed cube.  This result itself
followed a series of earlier similar results (see e.g.~\cite{Struik})
and was later extended by Dyson, Floyd, and others.

Jerrard's proof was a model of our proof in Section~\ref{s:triangles}.
He similarly considers a curve~$U$ on a torus~$T$, corresponding to
inscribed right isosceles triangles.  He then uses a parity
argument to conclude that~$U$ is not null homotopic, and a
separate argument to conclude that when moving along~$U$ the fourth
vertex cannot stay on the same side of the polygon.  Our approach has
several advantages due to the fact that we can make them generic and
thus avoid squares which have to be double counted.  Also, we use
a straightforward ad hoc argument with the minimal inscribed triangle,
different from that by Jerrard.  Overall, most details are still
different due to the different nature of intersections of analytic
and piecewise linear curves.

\subsection{} \.  In recent years, further results on the square
peg problem have appeared, most notably~\cite{Stormquist} and~\cite{Gri},
which both weakened the restrictions on the curves and extended the
reach of the theorem (to certain space quadrilaterals in~\cite{Stormquist}
and to rectangles in~\cite{Gri}).  In fact, there is a long history of
variations on the problem, which goes back to~\cite{Kakeya}.  Let us
mention some of them.

First, there are several results on inscribed triangles and rhombi and
rectangles in general Jordan curves~\cite{Nie-rhombi, NW}.  We refer
to~\cite{Nie-webpage} for the survey and further references.  Second, there
are several results on cyclic quadrilaterals inscribed into sufficiently smooth
curves~\cite{Makeev-cyclic,Makeev-two,Makeev}.  Note that in the piecewise linear
case, unless a quadrilateral~$Q \ssu \rr^2$ is an isosceles trapezoid, one can
always take a sufficiently slim triangle~$X$, such that no polygon similar to~$Q$
is inscribed into~$X$.  The corresponding ``isosceles trapezoid peg problem''
is open for general piecewise linear curves.  We believe that our proof by
deformation might be amenable to prove this result, but not without a
major change.

In a different direction, an interesting ``table theorem'' in~\cite{Fenn}
says that every sufficiently nice function~$f$ on a convex set~$U \ssu \rr^2$
has an an inscribed square of given size, defined as four points in~$U$
which are vertices of a square and have equal value of~$f$.  If the graph of~$f$
is viewed as a two-dimensional hill, the inscribed square can be interpreted as
feet of a square table, thus the name.  Note that when the curve~$C$ (in
the square peg problem) is a starred region, applying the table theorem to
the cone over~$C$ gives the desired square inscribed into~$C$.  We refer
to~\cite{KK,Mey,Mey-balancing,Mey-table} for more on the table theorem and
other related results.

Finally, there is a large number of results extending the square peg problem
to higher dimensions, including curves (see e.g.~\cite{Wu})
and surfaces (see e.g.~\cite{HLM,Kramer}).  These results are too numerous
to be listed here.  We refer to surveys~\cite[Section~B2]{CFG} %\cite{Klee}
and~\cite[Problem~11]{KW} for further references.

\subsection{} \.  In conclusion, let us mention that although stated differently,
the results for many classes of curves are essentially equivalent.  We already
saw this phenomenon in both proofs, where we applied what we called the
\emph{limit argument}.  In each case, we obtained one polygon as the limit
of others and noted that the sizes of inscribed squares do not converge to zero.
Of course, this approach fails in general, e.g.~a rectifiable curve can be obtained
as the limit of piecewise linear curves, but a priori the inscribes squares can
collapse into a point.

Nonetheless, one can use the limit argument in may cases that appear in the
literature.  It is easy to derive the square peg problem for analytic curves
from that of piecewise linear curves.  Similarly, the piecewise linear curves
can be obtained a a limit of analytic curves and derive our main theorem
from Jerrard's paper.  It would be interesting to see how far the limit
arguments take use from the piecewise linear curves.

\vskip.6cm

\noindent
{\bf Acknowledgements.} \  We are very grateful to Raman Sanyal
for listening to the first several versions of these proofs,
and to Ezra Miller for the encouragement.  Special thanks to
Elizabeth Denne for the interesting discussions and several
helpful references.  \ts The author was partially supported
by the NSF grant.

\newpage

%%%%%%%%%%%%%%%%%%%%%%%%%%%%%%%%%%%%%%%%%%%%%%%%%%%%%%%%%%%%%%%%%%%%%%%%

%%%%%%%%%%%%%%%%%%%%%%%%%%%%%%%%%%%%%%%%%%%%%%%%%%%%%%%%%%%%%%%%%%%%%%%%
\end{document}